\documentclass[11pt,a4paper]{amsart}
\usepackage{amssymb}
\usepackage{amsmath}
\usepackage[active]{srcltx}
\usepackage{t1enc}
\usepackage[latin2]{inputenc}
\usepackage{verbatim}
\usepackage{amsmath,amsfonts,amssymb,amsthm}
\usepackage[mathcal]{eucal}
\usepackage{enumerate}
\usepackage[centertags]{amsmath}
\usepackage{graphics}

\usepackage{url}  % ?? ???????: \usepackage{hyperref}

\numberwithin{equation}{section}

\newtheorem{theorem}{Theorem}[section]

\newtheorem{lemma}{Lemma}[section]
\newtheorem{remark}{Remark}[section]
\newtheorem{definition}{Definition}[section]

\newtheorem{proposition}{Proposition}[section]

\newtheorem*{Banach}{\textbf{Theorem A.}}
\newtheorem*{rem}{\textbf{Theorem B.}}

\begin{document}
\author{L.-E. Persson, V. Tsagareishvili, G. Tutberidze and G. Cagareishvili}

\title[Linear functionals]{Some New Sharp Boundedness Result Concerning Linear functionals on Banach spaces}
	
	\address{Lars-Erik Persson, UiT -- The Arctic University of Norway, Narvik, Norway \\ Karlstad University, Department of Mathematics, Karlstad,Sweden \\ Uppsala University, Department of Mathematics, Uppsala, Sweden}
	\email{larserik6pers@gmail.com}
		
	\address {Vakhtang Tsagareishvili, Ivane Javakhishvili Tbilisi State University, Department of Mathematics, Faculty of Exact and Natural Sciences, Chavchavadze str. 1, Tbilisi 0128, Georgia}
	\email{cagare@ymail.com}
	
	\address{Giorgi Tutberidze, The University of Georgia, Viktor Kupradze Institute of Mathematics, 77a Merab Kostava St, Tbilisi 0128, Georgia and Ivane Javakhishvili Tbilisi State University, Faculty of Exact and Natural Sciences, Chavchavadze str. 1, Tbilisi 0128, Georgia}
	\email{g.tutberidze@ug.edu.ge, g.tutberidze@tsu.ge}
	
	\address {Giorgi Cagareishvili, Ivane Javakhishvili Tbilisi State University, Department of Mathematics, Faculty of Exact and Natural Sciences, Chavchavadze str. 1, Tbilisi 0128, Georgia}
	\email{giorgicagareishvili7@gmail.com}

	\thanks{The work was supported by the Shota Rustaveli National Science Foundation grant YS-25-169.}

	\date{}
	\maketitle
	
	\begin{abstract}
		In this paper we state, prove, and discuss some new boundedness results concerning sequences of linear functionals of functions from the Lipschitz class \(Lip_1\) and general orthonormal systems. It is also proved that these results are, in a sense, the best possible.
		
	\end{abstract}
	
	\date{}
	
	

\section{Introduction}\label{sec1}

Concerning notations used in this introduction, we refer to our Section \ref{sec2}. The convergence or summability of Fourier series with respect to classical ONS - such as the trigonometric system (see e.g., \cite[Ch. 4]{kashin1999}), the Haar system (see e.g., \cite[Ch. 3]{kashin1999}), and the Walsh system (see \cite[Ch. 4]{kashin1999}) - for differentiable functions such problems are well studied and described as standard in the literature. This is primarily due to the well-understood structure and good localization properties of these systems, which align well with the smoothness of the functions under consideration.

However, the situation is markedly different when dealing with general ONS. Indeed, as shown by Banach's classical result \cite{banach1940}, the Fourier series with respect to certain ONS may fail to be summable even for simple functions such as \( f(x) = 1 \) on the interval \([0,1]\). This highlights the fact that, in general, summability cannot be taken for granted, even in the case of very regular and even constant functions.

In this paper, we investigate some linear functionals with special focus on the Banach space of functions from the Lipschitz class (\(C_L\)). By studying the boundedness of the sequence of these functionals, we will find conditions under which the Cesaro sums of general Fourier series of functions with derivatives from Lipschitz class functions (\(C_L\)) with respect to general orthonormal systems are bounded. We also note that in general, the Cesaro sums of functions of this class are not bounded. We also prove that the received results are the best possible in a special sense.

In the present work, we state, prove, and discuss some such results for those ONS \((\varphi_n)\) for which the condition
\[
\sigma_n^\alpha(x,f) = O(1)
\]
holds for every function \( f \) whose derivative \( f' \) belongs to the Lipschitz class \(Lip_1\), at a given point \( x \in G \). Similar problems have been investigated in some recent works (see \cite{PTW, person2023, tsagareishvili2021analysis, tsagareishvili2022smooth, tsagareishvili2022problems, tsagareishvili2022lipvar, tsagareishvili2022slovaca, Shalukhina}), where related problems of summability and convergence for various ONS are considered.

This paper is organized as follows: The main results are presented in Section \ref{sec3}. The main boundedness results are stated in Theorems \ref{theorem1} and \ref{theorem3} while the corresponding sharpness results are stated in Theorems \ref{theorem2} and \ref{theorem4}. Theorem \ref{theorem5} is an additional existence result. The proofs are given in Section \ref{sec4}, and Section \ref{sec5} is reserved for some final remarks and results. In order not to disturb the presentations and discussions in this section, some necessary preliminaries are given in Section \ref{sec2}. Some Lemmas in these sections are new and of independent relevance. Finally, in own conclusion section we sum up our new results and point out some complementary new possibilities for research in this fascinating area.

\section{Auxilary definitions and results} \label{sec2}
Let \( C_L \) be the class of functions $f$ such that \( \frac{df}{dx} \in \mathrm{Lip}_1 \).

The space \( \mathrm{Lip}_1 \) is a Banach space with the norm
\[
\|f\|_{\mathrm{Lip}_1} = \|f\|_{C} + \max_{x \ne y} \frac{|f(x) - f(y)|}{|x - y|},
\]
where  \(C[0,1]\) is a space of continuous functions, with the following norm:
\begin{equation*}
	\left\Vert f\right\Vert _{c}=\max_{x\in [0,1]}\left\vert f\left( x\right)
	\right\vert .
\end{equation*}

Suppose that \( (\varphi_n) \) is an orthonormal system (ONS) on \([0,1]\), and let \( f \in L_2 \) be an arbitrary function. The numbers
\[
C_n(f) = \int_0^1 f(x) \varphi_n(x)\,dx, \quad n = 1, 2, \dots
\]
are called the Fourier coefficients of the function \( f \) with respect to the ONS \( (\varphi_n) \).

The Fourier series of the function \( f \) is given by the expression:
\begin{equation}
	\sum_{k=1}^{\infty} C_k(f)\, \varphi_k(x).
	\label{2.1}
\end{equation}

The general partial sum of the series \eqref{2.1}, denoted by \( S_n(x, f) \), $n \in \mathbb{Z}_+,$ is given by the expression
\begin{equation}
	S_n(x,f) = \sum_{k=1}^{n} C_k(f)\, \varphi_k(x).
	\label{2.2}
\end{equation}

The Ces\'aro sums (\cite{alexits1961}, Ch.2) of order \( \alpha \) of the series \eqref{2.1}, denoted by \( \sigma_n^{\alpha}(x, f), \) are given, by the expression
\[
\sigma_n^{\alpha}(x, f) = \frac{1}{A_n^{\alpha}} \sum_{k=0}^{n-1} A_{n-k}^{\alpha}\, a_k\, \varphi_k(x), \]
where  \(\alpha > 0,\) and
\[
A_n^{\alpha} = \frac{(1+\alpha)(2+\alpha)\cdots(n+\alpha)}{n!}.
\]
For \( \alpha > 0 \), the following estimate holds:
\[
C_1 n^{\alpha} < A_n^{\alpha} < C_2 n^{\alpha} \ \ \ (C_1 \ \text{and} \ C_2  \ \text{  are positive constants}).
\]

We define the function \( g_n(x) \) by
\[
g_n(x) = \int_0^x \varphi_n(u)\,du,
\]
and \( Q_n(u, x) \) by
\[
Q_n(u,x) = \frac{1}{A_n^{\alpha}} \sum_{k=0}^{n-1} A_{n-k}^{\alpha} \, g_k(u) \, \varphi_k(x).
\]

We denote
\[
D_{nk}^{\alpha} = \frac{A_{n-k}^{\alpha}}{A_n^{\alpha}},
\]
and then the function \( Q_n(u,x) \) can be represented as
\begin{equation}
	Q_n(u,x) = \sum_{k=0}^{n-1} D_{nk}^{\alpha} g_k(u) \varphi_k(x).
	\label{2.3}
\end{equation}

Moreover, let
\begin{equation*}
	H_n(x) = \frac{1}{n}\sum_{i=1}^{n-1} \left| \int_0^{\frac{i}{n}} Q_n(u,x)\,du \right|.
\end{equation*}

We also define the sequence of linear functionals \(\{B_n(f,x)\}\) on the Banach space \(Lip_1 \) as follows:
\begin{equation*}
	B_n(f,x) := \int_{0}^{1} f(u) Q_n(u,x)du.
\end{equation*}    

\begin{lemma}\label{lemma1}
	\textit{Let \( (\varphi_k) \) be an arbitrary ONS. Then}
	\begin{equation}
		\frac{1}{n^2} \sum_{k=1}^{n} \varphi_k^2(x) = O(1)n^{-\frac{1}{2}}, \quad \text{a.e. on } [0,1].
		\label{2.5}
	\end{equation}	
\end{lemma}
\begin{proof}
	Since \[
	\frac{1}{n^2} \sum_{k=1}^{n} \varphi_k^2(x) \le \frac{1}{\sqrt{n}} \sum_{k=1}^{n} k^{-\frac{3}{2}} \varphi_k^2(x),
	\]
	we can conclude that
	\[
	\frac{1}{\sqrt{n}} \sum_{k=1}^{\infty} k^{-\frac{3}{2}} \int_0^1 \varphi_k^2(x)\,dx = \frac{1}{\sqrt{n}} \sum_{k=1}^{\infty} k^{-\frac{3}{2}} < +\infty,
	\]
	and it follows from Levi's theorem that the convergence holds almost everywhere, which proves that \eqref{2.5} holds.
\end{proof}

Let  \( G = [0,1] \setminus F \), where
\[
F = \left\{ x \in [0,1] \,\middle|\, \lim_{n \to \infty} \frac{1}{n^2} \sum_{k=1}^{n} \varphi_k^2(x) = \infty \right\}.
\]
It is easy to show that
\[
|F| = 0 \quad \text{and} \quad |G| = 1.
\]

\medskip

\begin{definition}
	We say that \( f \in E(x) \) if, at the point \( x \in [0,1] \),
	\[
	\limsup_{n \to \infty} \left| B_n(f, x) \right| < +\infty \ \ \ \ \ \ 
	\text{or} \ \ \ \ \ \ B_n(f, x)=O(1).\]
\end{definition}

\begin{lemma} 
	\label{Lemma2} 
	For every integer \( i \) with \( 1 \leq i \leq n \) and for all \( x \in [0,1] \), the following inequality holds (see \eqref{2.3}):
	\begin{equation}
		\int_{\frac{i-1}{n}}^{\frac{i}{n}} |Q_n(u, x)|\,du \leq \left( \frac{1}{n^2} \sum_{k=1}^n \varphi_k^2(x) \right)^{\frac{1}{2}}.
		\label{2.6}
	\end{equation}
\end{lemma}

\begin{proof}
	From the Bessel inequality, it follows that
	
	\begin{equation}
		\sum_{k=1}^\infty g_k^2(u) = \sum_{k=1}^\infty \left( \int_0^u \varphi_k(t)\,dt \right)^2 \le \int_0^u dt \le 1.
		\label{2.7}
	\end{equation}
	Here, the function \( g_k \) is defined by
	\[
	g_k(u) = \int_0^u \varphi_k(t)\, dt, \quad u \in [0,1].
	\]
	
	By applying the Cauchy and H\"older inequalities, as well as inequality \eqref{2.7}, and taking into account that \( D_{nk}^{\alpha} < 1 \), we can conclude that, for each \( i = 1, \dots, n \),
	
	\begin{align*}
		\int_{\frac{i-1}{n}}^{\frac{i}{n}} |Q_n(u,x)|\,du
		&\le \frac{1}{\sqrt{n}} \left( \int_{\frac{i-1}{n}}^{\frac{i}{n}} Q_n^2(u,x)\,du \right)^{\frac{1}{2}} \\
		& = \frac{1}{\sqrt{n}} \left( \int_{\frac{i-1}{n}}^{\frac{i}{n}} \left( \sum_{k=1}^n D_{nk}^{\alpha} g_k(u) \varphi_k(x) \right)^2 du \right)^{\frac{1}{2}}\\
		&\le \frac{1}{\sqrt{n}} \left( \int_{\frac{i-1}{n}}^{\frac{i}{n}}  \sum_{k=1}^n (D_{nk}^{\alpha})^2 g_k^2(u)\, du \cdot \sum_{k=1}^n \varphi_k^2(x) \right)^{\frac{1}{2}} \\
		&\le \frac{1}{n} \left( \sum_{k=1}^n \varphi_k^2(x) \right)^{\frac{1}{2}}.
	\end{align*}
	The proof is complete.
	
\end{proof}

\begin{lemma}
	\label{lemma3}
	Let   \( (\varphi_n) \) be an ONS on \([0,1]\), and let \( f \in C_L \). Then
	\begin{equation}
		\sigma_n^{\alpha}(x, f) = f(1)\, \sigma_n^{\alpha}(x, q) - \int_0^1 f'(u)\, Q_n(u,x)\,du,
		\label{2.8}
	\end{equation}
	where \( q(u) = 1 \) for all \( u \in [0,1] \).
	
\end{lemma}

\begin{proof}
	By integrating by parts, we have that
	\[
	C_k(f) = \int_0^1 f(u)\, \varphi_k(u)\,du
	= f(1) \int_0^1 \varphi_k(u)\,du - \int_0^1 f'(u)\, g_k(u)\,du,
	\]
	that is,
	\begin{equation}
		C_k(f) = f(1) C_k(q) - \int_0^1 f'(u)\, g_k(u)\,du.
		\label{2.9}
	\end{equation}
	Therefore, by \eqref{2.9}
	\begin{align*}
	\sum_{k=0}^{n-1} D_{nk}^{\alpha} C_k(f)\, \varphi_k(x)
& = f(1) \sum_{k=0}^{n-1} D_{nk}^{\alpha} C_k(q)\, \varphi_k(x)
	\\
	& \qquad - \int_0^1 f'(u) \sum_{k=0}^{n-1} D_{nk}^{\alpha} g_k(u) \varphi_k(x)\, du.
	\end{align*}
	By also taking equation \eqref{2.3} into consideration, we obtain the desired result \eqref{2.8}. The proof is complete.
	
\end{proof}

We also need the following well-known results:
\begin{Banach}(see \cite{banach1940})  \label{Banach}
	\textit{ \ For any \( f \in L_2 \),  \( \left(f \not\sim 0 \right) \), there exists an ONS \( (\varphi_n) \) such that}
	\[
	\limsup_{n \to \infty} \left| \sigma_n^{\alpha}(x, f) \right| = +\infty, \quad \text{a.e. on } [0,1].
	\]
\end{Banach}

\begin{rem} (see \cite{gogoladze2012})  \label{BBBB}
	\textit{ \ If \( F \in L_2 \) and \(f\) is a finite function in any point on \([0,1]\),  then the following identity holds:}
	\begin{align}
		\label{2.10}\int_0^1 f(u)F(u)\,du & = \sum_{i=1}^{n-1} \left(f\left(\frac{i}{n}\right) -f\left(\frac{i+1}{n}\right)\right) \int_{0}^{\frac{i}{n}} F(u)du \\
		&+\sum_{i=1}^{n-1} \int_{\frac{i-1}{n}}^{\frac{i}{n}} \left(f(u) - f\left(\frac{i}{n}\right)\right) F(u) \, du
		+ f(1) \int_0^1 F(u) \, du.
		\notag
	\end{align}
\end{rem}

\section{Statement of the Main Results}\label{sec3}

Our first main result reads:
\begin{theorem}\label{theorem1}
	Let \( B_n(f,x) \) be an arbitrary sequence of linear functionals on \( Lip_1 \), and let \( q \in E( x) \), where  \( q(u) = 1 \) for all \( u \in [0,1] \).
	If, for a point \( x \in G \), the condition  
	\begin{equation}
		H_n(x) = O(1)
		\label{eq:3.1}
	\end{equation}
	holds, then for any function \( f \in Lip_1 \), the following estimate is valid at the point \( x \in G \):
	\[ B_n(f,x) = O(1).	\]
\end{theorem}
The sharpness of Theorem \ref{theorem1} is shown in our next theorem.

\begin{theorem} \label{theorem2}
	Let \( (B_n(f,x)) \)) be an arbitrary sequence of linear functionals on \(Lip_1\). If for some point \( x_0 \in G \) the following holds:
	\[
	\limsup_{n \to \infty} |H_n(x_0)| = +\infty,
	\]
	then there exists a function \( r \) such that \( r \in Lip_1 \) and
	\begin{equation} \label{eq:3.3}
		\limsup_{n \to \infty} \left| B_n(r, x_0) \right| = +\infty.
	\end{equation}
\end{theorem}
Next we investigate the systems for which
\[
\limsup_{n \to \infty} \left| \sigma_n^{\alpha}(x,f) \right| < +\infty.
\]
holds for all functions $x\in G$ with $f'\in \mathrm{Lip_1}$. First, we state:
\begin{theorem}
	\label{theorem3}
	Let \((\varphi_n)\) be an ONS on \([0; 1]\) and \(q \in E(x)\) ($q=1, \ u\in [0,1] $). If, for some point \(x \in G\), the condition
	\[\left|H_n(x)\right| = O(1)\]
	holds,  then for any function \( f' \in Lip_1 \) the following estimate is valid at any point \( x \in G \):
	\[ \left|\sigma_n^{\alpha}(x,f)\right| = O(1).	\]
\end{theorem}

The sharpness of Theorem \ref{theorem3} is described in the following sense:
\begin{theorem}
	\label{theorem4}
	Let \((\varphi_n)\) be an ONS on \([0; 1]\). If, for some point \(x_0 \in G\), the condition
	\begin{equation}
		\label{eq3.17} \limsup_{n \to \infty}\left|H_n(x_0)\right| = +\infty
	\end{equation}
	holds, then there exists a function \(r\),  \( r' \in Lip_1 \) such that the following estimate is valid:
	\[ \limsup_{n \to \infty} \left|\sigma_n^{\alpha}(x_0,r)\right| = +\infty.	\]
\end{theorem}
Finally, we state the following related existence result:
\begin{theorem}
	\label{theorem5}
	There exists an ONS \( (\vartheta_n) \) and a function \( g \in C_L \), such that
	\begin{itemize}
		\item[(a)] \(\displaystyle \int_0^1 \vartheta_n(u) \, du = 0, \quad n = 1, 2, \ldots\)
		\vspace{0.3cm}
		\item[(b)] \(\displaystyle \limsup_{n \to \infty} \left| \sigma_n^\alpha(x, g, \vartheta) \right| = +\infty, \quad \text{a. e. } x \in \left[0, 1\right]\),
	\end{itemize}
	where
	\begin{equation*}
		\sigma_n^\alpha(x, g, \vartheta) = \sum_{k=1}^{n} D_{nk}^{\alpha} C_n(g,\vartheta)\vartheta_n(x),
	\end{equation*}
	and
	\begin{equation*}
		C_n(g, \vartheta)= \int_0^1 g(u) \vartheta_n(u) \, du, \quad n = 1,2, 3 \cdots .
	\end{equation*}
\end{theorem}

\section{Proofs of the Theorems}\label{sec4}

\begin{proof}[Proof of Theorem \ref{theorem1}]
	
	If we substitute \( F(x) = Q_n(u,x) \) into the equation \eqref{2.10} in Theorem B, we obtain that
	\begin{align}
		\int_0^1 f(u) Q_n(u, x) \, du &= 
		\sum_{i=1}^{n-1} \left( \left(f \left(\frac{i}{n} \right) - f\left(\frac{i+1}{n} \right) \right) 
		\int_0^{\frac{i}{n}} Q_n(u, x) \, du \right) \notag \\
		\label{eq:3.2} &\quad + \sum_{i=1}^{n-1} \int_{\frac{i-1}{n}}^{\frac{i}{n}} \left(f(u) - f \left(\frac{i}{n} \right)\right) Q_n(u, x) \, du \\
		& \quad + f(1) \int_0^1 Q_n(u, x) \, du  := I_1 + I_2 + I_3. \notag
	\end{align}
	Morover, by using \eqref{eq:3.1} and the fact that  \(f\in Lip_1,\) we find that
	\begin{align*}
		|I_1| &=  O\left(\tfrac{1}{n}\right) \sum_{i=1}^{n-1}  \left| \int_0^{\frac{i}{n}} Q_n(u, x) \, du \right|   = O(1)	H_n(x) = O(1).
	\end{align*}
	
	Next, since \( f\in Lip_1 \) and \( x \in G \), by applying inequality \eqref{2.7} from Lemma \ref{Lemma2} and Lemma \ref{lemma1} \( \left( D_{nk}^\alpha < 1\right) \), we obtain that
	\begin{align*}
		|I_2| 
		&\leq  O\left( \tfrac{1}{n} \right) \sum_{i=1}^n \int_{\frac{i-1}{n}}^{\frac{i}{n}} |Q_n(u, x)| \, du = O(1) \cdot \tfrac{1}{n} \int_0^1 |Q_n(u, x)| \, du \notag \\
		&= O(1) \cdot \tfrac{1}{n} \left( \int_0^1 Q_n^2(u, x) \, du \right)^{\frac{1}{2}} = O(1) \cdot \tfrac{1}{n} \left( \int_0^1 \left( \sum_{k=1}^n D_{nk}^{\alpha} g_k(u) \varphi_k(x) \right)^2 du \right)^{\frac{1}{2}} \notag \\
		&= O(1) \cdot \tfrac{1}{n} \left( \int_0^1 \sum_{k=1}^n g_k^2(u) du \sum_{k=1}^n \varphi_k^2(x) \right)^{\frac{1}{2}} \leq O(1) \cdot \tfrac{1}{n} \left( \sum_{k=1}^n \varphi_k^2(x) \right)^{\frac{1}{2}} \notag \\
		& = O(1) \cdot \left( \tfrac{1}{n^2} \sum_{k=1}^n \varphi_k^2(x) \right)^{\frac{1}{2}} = O(1) \notag .
	\end{align*}
	Furthermore, if \(q(u) = 1, \ u\in[0,1],\) according to the conditions of Theorem A, where \( q \in E(x), \) we get that
	\begin{equation*}
		|I_3| \leq \left|f(1) \int_0^1 Q_n(u, x) du \right| = \left|f(1)\right|\cdot \left| \int_0^1 q(u) Q_n(u,x) \, du \right| = \left|f(1)\right|\cdot\left|B_n(q,x) \right|= O(1).
	\end{equation*}
	
	Finally, taking into account the estimates above of \(\left| I_1 \right|, \left| I_2 \right| \text{ and } \left| I_3 \right|
	\) in \eqref{eq:3.2} we find that, for any \(f\in Lip_1,\) we can conclude that
	\begin{equation}
		\label{3.3} B_n(f,x) =\int_0^1 f(u) Q_n(u, x) du = O(1),
	\end{equation}
	and the proof is complete.
	
\end{proof}

\begin{proof}[Proof of Theorem \ref{theorem2}] First, we suppose that 
	\begin{equation*}
		\limsup_{n \to \infty} \left| \int_0^1 Q_n(u, x_0) \right| = +\infty.
	\end{equation*} 
	Then (see \eqref{3.3} with \(f=q\))
	\begin{equation*}
		\limsup_{n \to \infty} \left| B_n(q, x_0) \right| = \limsup_{n \to \infty} \left| \int_0^1 Q_n(u, x_0) \right|=+\infty.
	\end{equation*}
	and since \(q\in Lip_1\), \eqref{eq:3.3} holds for this case.
	
	Next, we assume that
	\begin{equation}
		\label{eq3.5} \left| B_n (q, x_0) \right| = O(1).
	\end{equation}
	
	On the space \( \mathrm{Lip}_1 \), we consider the sequence of functions \( (r_n) \), defined as follows:
	
	\begin{equation}\label{eq3.6}
		r_n(u) = \int_0^u \operatorname{sign} \left( \int_0^y Q_n(v, x_0) \, dv \right) dy, \quad n = 1, 2, \ldots
	\end{equation}
	
	The following equality is valid (see \eqref{eq:3.2}, with \(f=r_n\))
	\begin{align}\label{eq3.7}
		\int_0^1 r_n(u) Q_n(u, x_0) \, du 
		&= \sum_{i=1}^{n-1} \left( r_n\left(\frac{i}{n}\right) - r_n\left(\frac{i+1}{n}\right) \right) \int_0^{\frac{i}{n}} Q_n(u, x_0) \, du \notag \\
		&\quad + \sum_{i=1}^{n-1} \int_{\frac{i-1}{n}}^{\frac{i}{n}} \left( r_n(u) - r_n\left(\frac{i}{n}\right) \right) Q_n(u, x_0) \, du  \\
		&\quad + r_n(1) \int_0^1 Q_n(u, x_0) du := A_1 + A_2 + A_3. \notag
	\end{align}
	
	Next, we proceed to by estimating the quantities \( |A_1| \), \( |A_2| \), and \( |A_3| \).
	
	By equation \eqref{2.6} (see  Lemma \ref{Lemma2}) we have that
	\begin{align}\label{al1}
		|A_2| &\leq \frac{1}{n} \sum_{i=1}^{n-1} \int_{\frac{i-1}{n}}^{\frac{i}{n}} |Q_n(u, x_0)| \, du \leq \frac{1}{n} \int_0^1 |Q_n(u, x_0)| \, du\\
		&= O(1)\frac{1}{n} \left(\sum_{k=1}^n \varphi_k^2(x_0) \right)^{\frac{1}{2}} = O(1).	\notag
	\end{align}
	Moreover, by taking into account  \eqref{eq3.5} and \eqref{eq3.6} we get that
	\begin{equation} \label{eq2}
		|A_3| = r_n(1) \cdot \left| \int_0^1 Q_n(u, x_0) \, du \right| = O(1) \cdot B_n(q,x_0) = O(1).
	\end{equation}
	
	By \( E_n \) we denote the set of all indices \( i \), \( i \in \{1, 2, \ldots, n-1\} \), for which there exists a point \( t \in \left[\frac{i}{n}, \frac{i+1}{n}\right) \) such that
	\[
	\operatorname{sign} \int_0^{\frac{i}{n}} Q_n(u, x_0) \, du \neq \operatorname{sign} \int_0^t Q_n(u, x_0) \, du.
	\]
	
	If \( i \in E_n \), then, due to the continuity of the function	\( \int_0^t Q_n(u, x_0) \, du\) on the interval \([0,1]\), there exists some \( t_{in} \in \left[\frac{i}{n}, \frac{i+1}{n}\right) \) such that
	
	\[
	\int_0^{t_{in}} Q_n(u, x_0) \, du = 0.
	\]
	Consequently, we have that
	\[
	\int_0^{\frac{i}{n}} Q_n(u, x_0) \, du = \int_{t_{in}}^{\frac{i}{n}} Q_n(u, x_0) \, du.
	\]
	For this reason, by \eqref{2.7}, we find that
	\begin{align*}
		&\sum_{i \in E_n} \left| \int_0^{\frac{i}{n}} Q_n(u, x_0) \, du \right|
		\leq \sum_{i=1}^{n-1} \left| \int_{t_{in}}^{\frac{i}{n}} Q_n(u, x_0) \, du \right| \leq \int_0^1 |Q_n(u, x_0)| \, du \notag \\
		&= \left( \int_0^1 Q_n^2(u, x_0) \, du \right)^{\frac{1}{2}} = \left( \int_0^1 \left( \sum_{k=1}^n D_{nk}^\alpha g_k(u) \varphi_k(x_0) \right)^2 du \right)^{\frac{1}{2}} \notag \\
		&\leq\left( \int_0^1 \sum_{k=1}^n g_k^2(u) \, du \cdot \sum_{k=1}^n \varphi_k^2(x_0) \right)^{\frac{1}{2}} \leq \left( \sum_{k=1}^n \varphi_k^2(x_0) \right)^{\frac{1}{2}}.
	\end{align*}
	From here and Lemma \ref{lemma1}, we can conclude that 
	\begin{align}\label{eq:3.10}
		\frac{1}{n} \sum_{i \in E_n} \left| \int_0^{\frac{i}{n}} Q_n(u, x_0) \, du \right| = O(1)\frac{1}{n} \left( \sum_{k=1}^n \varphi_k^2(x_0) \right)^{\frac{1}{2}} = O(1).
	\end{align}
	
	Next we define the set \( F_n := \{1, 2, \ldots, n-1\} \setminus E_n \). Then, for any \( i \in F_n \), by the definition of \( E_n \), the following identity holds:
	\[
	\left( r_n\left(\frac{i}{n}\right) - r_n\left(\frac{i+1}{n}\right) \right) \int_0^{\frac{i}{n}} Q_n(u, x_0) \, du = -\frac{1}{n} \left| \int_0^{\frac{i}{n}} Q_n(u, x_0) \, du \right|.
	\]
	Therefore, by using \eqref{eq3.6} it immediately follows that
	\begin{align}\label{eq:3.11}
		\left| \sum_{i \in F_n} \left( r_n \left(\frac{i}{n}\right) - r_n \left(\frac{i+1}{n}\right) \right) \int_0^{\frac{i}{n}} Q_n(u, x_0) \, du \right|
		=  \frac{1}{n} \sum_{i \in F_n} \left| \int_0^{\frac{i}{n}} Q_n(u, x_0) \, du \right| 
	\end{align}
	\begin{align}
		& = -\frac{1}{n} \sum_{i=1}^{n-1} \left| \int_0^{\frac{i}{n}} Q_n(u, x_0) \, du \right| - \frac{1}{n} \sum_{i \in E_n} \left| \int_0^{\frac{i}{n}} Q_n(u, x_0) \, du \right| \notag \\
		& = H_n(x_0) - \frac{1}{n} \sum_{i \in E_n} \left| \int_0^{\frac{i}{n}} Q_n(u, x_0) \, du \right|. \notag
	\end{align}
	On the other hand, if \( i \in E_n \), then (see \eqref{eq:3.10}) we have that
	\begin{align}\label{eq:3.12}
		\left| \sum_{i \in E_n} \left( r_n \left(\frac{i}{n}\right) - r_n \left(\frac{i+1}{n}\right) \right) \int_0^{\frac{i}{n}} Q_n(u, x_0) \, du \right| \leq \frac{1}{n} \sum_{i \in E_n} \left| \int_0^{\frac{i}{n}} Q_n(u, x_0) \, du \right| = O(1).
	\end{align}
	Hence, by combining \eqref{eq:3.11} and \eqref{eq:3.12} , we obtain that
	\begin{align}
		\label{eq:3.13}
		|A_1| &= \left| \sum_{i=1}^{n-1} \left( r_n\left(\frac{i}{n}\right) - r_n\left(\frac{i+1}{n}\right) \right) \int_0^{\frac{i}{n}} Q_n(u, x_0) \, du \right| \\
		&= \left| \sum_{i\in F_n} \left( r_n\left(\frac{i}{n}\right) - r_n\left(\frac{i+1}{n}\right) \right) \int_0^{\frac{i}{n}} Q_n(u, x_0) \, du\right. \notag \\
		& +\left. \sum_{i\in E_n} \left( r_n\left(\frac{i}{n}\right) - r_n\left(\frac{i+1}{n}\right) \right) \int_0^{\frac{i}{n}} Q_n(u, x_0) \, du \right| \geq H_n(x_0) - O(1). \notag
	\end{align}
	Finally, by combining \eqref{eq3.7}, \eqref{al1}, \eqref{eq2} and \eqref{eq:3.13}, we find that
	\begin{equation*}
		\left| \int_0^1 r_n(u) Q_n(u, x_0) \, du \right| \geq H_n(x_0) - O(1).
	\end{equation*}
	Hence, by keeping in mind the conditions of Theorem \ref{theorem2}, we can conclude that
	\begin{equation}\label{eq:3.14}
		\limsup_{n \to \infty} \left| \int_0^1 r_n(u) Q_n(u, x_0) \, du \right| = +\infty, \text{ or } \limsup_{n \to \infty} \left| B_n(r_n,x_0 \right| = +\infty.
	\end{equation}
	
	Moreover, it is easy to see that (\(x \neq y\))
	\begin{equation}\label{eq:3.15}
		\|f_n\|_{\mathrm{Lip}_1} = \|f_n\|_C + \sup_{x,y \in [0,1]} \frac{|f_n(x) - f_n(y)|}{|x - y|} \leq 2.
	\end{equation}
	Therefore, according to the Banach-Steinhaus Theorem, \eqref{eq:3.14} and \eqref{eq:3.15}, we can conclude that there exists a function \( r \in \mathrm{Lip}_1 \) such that
	\begin{equation}\label{eq:3.16}
		\limsup_{n \to \infty} \left| \int_0^1 B_n(r, x_0) \, du \right| = +\infty.
	\end{equation}
	The proof is complete.
	
\end{proof}

\begin{proof}[Proof of Theorem \ref{theorem3}]
	According to Theorem \ref{theorem1}, we find that
	\begin{equation}
		\label{eq3.15} B_n(f',x) = O(1)
	\end{equation}
	for any \(f' \in Lip_1\). Moreover, by using equation \eqref{2.8} from Lemma \ref{lemma3}, we find that
	\begin{equation*}
		\sigma_n^{\alpha}(x,f) = f(1) \sigma_n^{\alpha}(x,q) -\int_0^1 f'(u) Q_n(u,x)du = f(1) \sigma_n^{\alpha}(x,q) - B_n(f',x),
	\end{equation*}
	where \(q(u) = 1, \ u \in [0,1]\). Therefore, since \(q \in E(\varphi,x)\), we obtain that
	\[\sigma_n^{\alpha}(x,q) = O(1).\]
	It follows that
	\begin{equation}
		\label{eq16}\sigma_n^{\alpha}(x,f) = \sigma_n^{\alpha}(x,q) -B_n(f',x),
	\end{equation}
	and the assumption $q \in E(x)$, together with \eqref{eq3.15}, imply that
	\[\sigma_n^{\alpha}(x,f) = O(1).\]
	and the proof is complete.
\end{proof}

\begin{proof}[Proof of Theorem \ref{theorem4}]
	First, we note that 
	\begin{equation}
		\label{eq3.18} \sigma_n^{\alpha}(x_0,r) = O(1).
	\end{equation}
	otherwise, since \(r'\in lip_1\) Theorem \ref{theorem4} is proven.
	
	Next we suppose that \eqref{eq3.18} is fulfilled. From condition \eqref{eq3.17} and by Theorem \ref{theorem2}, we get that there exists a function \(r, \ r' \in Lip_1\) such that
	\begin{equation}
		\label{eq3.19} \limsup_{n \to \infty}\left|B_n(r',x_0)\right| = +\infty.
	\end{equation}
	Hence, by taking into account \eqref{eq16} when \(f=r\) we can conclude (see \eqref{2.8}, \eqref{eq3.18}, \eqref{eq3.19}) that
	\[ \limsup_{n \to \infty} \left|\sigma_n^{\alpha}(x_0,r)\right| = \limsup_{n \to \infty} \left|\sigma_n^{\alpha}(x_0,r)-B_n(r',x_0)\right|=+\infty.	\]
	The proof is complete.
	
\end{proof}

\begin{remark}
	This means that the condition of Theorem \ref{theorem3}, \(q\in E(x_0)\) (\(q(u)=1, u\in [0,1]\)), does not guarantee that
	\[ \limsup_{n \to \infty} \left|\sigma_n^{\alpha}(x_0,f)\right| < +\infty,	\]
	for any function \(f \in C_L.\) 
\end{remark}

\begin{proof}[Proof of Theorem \ref{theorem5}]
	We put 
	\[
	f(u) = 1 - \cos\big(4\pi(u - \frac{1}{2})\big).
	\]
	According to the Theorem A, there exists an ONS  \((\phi_n)\) such that, a. e. on  \([0,1]\),
	\begin{equation}\label{eq*}
		\limsup_{n \to \infty} \left| \sigma_n^{\alpha}(x, f, \varphi) \right| = +\infty.
	\end{equation}
	We define an ONS \((\vartheta_n)\)  as follows:
	\[
	\vartheta_n(u) = 
	\begin{cases}
		\varphi_n(2u), & \text{ when } u \in \left[0, \tfrac{1}{2}\right), \\
		-\varphi_n \left(2\left(u-\frac{1}{2} \right) \right), & \text{when } u \in \left[\tfrac{1}{2}, 1\right].
	\end{cases}
	\]              
	It is easy to prove that \((\vartheta_n)\)  is an ONS on \([0,1]\) and \(\int_0^1\vartheta_n (u)du = 0, \ \ n=0, 1, \dots .\)
	
	Thus, the ONS \((\vartheta_n)\) satisfies the conditions a) of Therefore \ref{theorem5}.
	
	Now we consider the function \( g(u) \) defined by
	\[
	g(u) = 
	\begin{cases}
		f(2u), & \text{for } u \in \left[0, \tfrac{1}{2}\right), \\
		0, & \text{for } u \in \left[\tfrac{1}{2}, 1\right).
	\end{cases}
	\]
	Then,
	\[
	g'(u) = 
	\begin{cases}
		4\pi \sin\left(4\pi(u - \tfrac{1}{2})\right), & \text{for } u \in \left[0, \tfrac{1}{2}\right), \\
		0, & \text{for } u \in \left[\tfrac{1}{2}, 1\right).
	\end{cases}
	\]
	Obviously \(g' \in Lip_1\) and 
	
	\begin{align*}
		C_n(g,\vartheta) &= \int_0^1g(u) \vartheta_n(u) du = \int_{0}^{\frac{1}{2}} f(2u) \varphi_n (2u) du
		C_n(g,\vartheta) = \int_0^1g(u) \vartheta_n(u) du \\
		&=\frac{1}{2}\int_{0}^{1} f(u) \varphi_n (u) du = \frac{1}{2} C_n (f, \varphi),
	\end{align*}
	and it's follows that
	\begin{equation*}
		\sigma_n^\alpha(x, g, \vartheta) =\frac{1}{2}\sigma_n^\alpha(x, f, \varphi).
	\end{equation*}
	Thus, in view of \eqref{eq*}, also the condition (b) of Theorem \ref{theorem5} is fulfilled. The proof is complete.
	
\end{proof}

\section{Final remarks and results}\label{sec5}
The condition (see \eqref{eq:3.1})
\[H_n(x) = O(1)\]
is said to be efficient if it is easily verified for classical ONS (trigonometric, Haar, Walsh, and other systems).

\begin{proposition}
	\label{theorem4.1} We define the sequence of functions \((\varphi_n)\) on the interval \([0,1]\) by
	\[
	\varphi_n(u) = \sqrt{2} \cos(2\pi n u).
	\]
	Then, for any \(x \in [0, 1]\), the following efficiency estimate holds:
	\[
	H_n(x) = O(1).
	\]
\end{proposition}

\begin{proof}
	By the Cauchy inequality and the fact that \( D_{nk}^\alpha < 1 \), we can estimate as follows:
	
	\begin{align*}
		T_n(x) &= \frac{1}{n} \sum_{i=1}^{n-1} \left| \int_0^{\frac{i}{n}} Q_n(u,x) \, du \right| = \frac{2}{n} \sum_{i=1}^{n-1} \left| \int_0^{\frac{i}{n}} \sum_{k=1}^n D_{nk}^{\alpha} \left( \int_0^u \cos(2\pi k v) \, dv \right) \cos(2\pi k x) \, du \right| \notag \\
		&= O(1) \cdot \max_{1 \leq i \leq n} \left| \int_0^{\frac{i}{n}} \frac{1}{2\pi} \sum_{k=1}^n D_{nk}^\alpha \cdot \frac{1}{k} \sin(2\pi k u) \cos(2\pi k x) \, du \right| \notag \\
		&= O(1) \cdot \left( \int_0^1 \left( \sum_{k=1}^n \frac{1}{k^2} \right) \, du \right)^{\frac{1}{2}} = O(1).
	\end{align*}
	
	The proof is complete.
\end{proof}

\begin{remark}
	It seems that our conclusions are also correct for other orthonormal systems. In Theorem \ref{theorem4.1}, we pointed out the fact that our results are also valid for trigonometric systems. However, our conclusions are also consistent for other systems. Next, we state and prove this fact for the Haar system.
	
\end{remark}

\begin{proposition}
	\label{theorem4.2}
	Let \( (X_n) \) be the Haar system. Then for any \( x \in [0,1] \), 
	\[
	H_n(x) = O(1).
	\]
\end{proposition}

\begin{proof}
	It is known that (see \cite{alexits1961, kashin1999}) if \( 2^s < m \leq 2^{s+1} \), then 
	\begin{equation}
		\label{eq5.00002}\left| \int_0^u X_m(v) \, dv \right| \leq 2^{-\frac{s}{2}}.
	\end{equation}
	From  \eqref{eq5.00002}, it follows that for any \( i = 1, 2, \dots, n-1 \) and any \( x \in [0,1] \), we have that
	
	\begin{equation*}
		\left| \int_0^{\frac{i}{n}} \int_0^u \sum_{m=2^s+1}^{2^{s+1}}  D_{nk}^{\alpha} X_m(v) X_m(x)  dv  du \right| \leq 2 \cdot 2^{-s}.
	\end{equation*}
	Thus, we obtain the following estimate:
	\begin{align*}
		T_n(x) 
		&= \frac{1}{n} \sum_{i=1}^{n-1} \left| \int_0^{\frac{i}{n}} Q_n(u, x) \, du \right|  \\
		&= O(1) \cdot \max_{1 \leq i < n} \left| \int_0^{\frac{i}{n}} \sum_{s=0}^{d} \sum_{m=2^s+1}^{2^{s+1}} D_{nk}^{\alpha} X_m(v) X_m(x) \, dv \right| = O(1).
	\end{align*}
	The proof is complete.
\end{proof} 

\begin{remark}
	This research opens up new possibilities for further exploration in functional analysis, particularly in understanding the interplay between Lipschitz continuity and Fourier series convergence. Future work may involve extending these findings to other classes of functions, potentially revealing deeper insights into the structure of function spaces
\end{remark}

\begin{remark}
	The recent book \cite{PTW} has its main focus on the Villenking-Fourier series. It contains a particular a number of open problems. This paper raises more interesting questions in this fascinating area by considering more general orthogonal systems.
\end{remark}

\section{Conclusion}

\qquad In this paper, we have investigated the boundedness properties of sequences of linear functionals acting on the Lipschitz class $Lip_1$, with a specific application to the summability of Fourier series with respect to general orthonormal systems. While classical systems possess structural and localization advantages that guarantee convergence for smooth functions, general ONS often exhibit pathological behaviors, as famously highlighted by Banach's classical divergence theorem. 

Our main contribution establishes a sharp characterization for the boundedness of these functionals and the corresponding Ces\`aro means $\sigma_n^\alpha(x,f)$. Specifically, we proved that the condition $H_n(x) = O(1)$ is both sufficient (Theorems \ref{theorem1} and \ref{theorem3}) and necessary (Theorems \ref{theorem2} and \ref{theorem4}) to ensure the bounded behavior of the respective mathematical objects. The construction of the specific ONS in Theorem \ref{theorem5} further underscores the structural limitations of general orthogonal systems and guarantees that our sharpness results hold almost everywhere. 

The methodology developed in this work opens up several natural directions for future research:
\begin{itemize}
	\item \textbf{Generalizations to other function classes:} A logical next step involves examining whether the sharp boundary established by the condition $H_n(x) = O(1)$ can be extended to broader generalized Lipschitz spaces $Lip_\beta$ ($0 < \beta < 1$) or to the spaces of functions of bounded variation.
	\item \textbf{Applications to alternative summability methods:} The geometric and analytic techniques applied herein to Ces\`aro means $\sigma_n^\alpha$ provide a robust framework that could be adapted to analyze the pointwise boundedness and localization properties of other classical matrix summability methods for general ONS.
\end{itemize}

In summary, the results presented herein enrich the classical theory of general orthogonal systems and provide sharp analytical tools for evaluating the limitations of linear approximation processes on smooth functional classes.


\begin{thebibliography}{9}
		\bibitem{alexits1961}
		G. Alexits,  Convergence Problems of Orthogonal Series, \textit{International Series of Monographs in Pure and Applied Mathematics}, (1961).
		
		\bibitem{banach1940}
		S. Banach,  Sur la divergence des s\'eries orthogonales, \textit{Stud. Math.}, 9 (1940), {139}--{155}. 
		
		\bibitem{gogoladze2012}
		L. Gogoladze and V. Tsagareishvili,  Fourier coefficients of continuous functions, \textit{Mat. Zametki}, 91 (2012), 5, {691}--{703}; English transl.: \textit{Math. Notes}, 91 (2012), 5-6, {645}--{656}.
		
		\bibitem{kashin1999}
		B. S. Kashin and A. A. Saakyan,  Orthogonal Series, 2nd ed., \textit{AFTs, Moscow}, (1999); English transl.: \textit{AMS, Transl. Math. Monogr.}, 75 (1989).
		
		\bibitem{PTW}
		L.-E. Persson, G. Tephnadze and F. Weisz,  Martingale Hardy Spaces and Summability of one-dimensional Vilenkin-Fourier Series, \textit{Birkh{\"a}user/Springer}, (2022).
		
		\bibitem{person2023}
		L.-E. Persson, V. Tsagareishvili and G. Tutberidze,  Properties of sequence of linear functionals on BV with applications, \textit{Nonlinear Studies}, 30 (2023), 4, {1319}--{1328}.
		
		\bibitem{Shalukhina}
		A. Shalukhina,  Self-improving boundedness of the maximal operator on quasi-Banach lattices over spaces of homogeneous type, \textit{J. Math. Anal. Appl.}, 548 (2025), 2, Article ID {129419}.
		
		\bibitem{tsagareishvili2021analysis}
		V. Tsagareishvili,  Some properties of general Fourier coefficients of Lipschitz functions, \textit{Anal. Math.}, 47 (2021), {229}--{241}.
		
		\bibitem{tsagareishvili2022smooth}
		V. Tsagareishvili,  Smooth functions and general Fourier coefficients, \textit{J. Contemp. Math. Anal.}, 57 (2022), 2, {102}--{111}.
		
		\bibitem{tsagareishvili2022problems}
		V. Tsagareishvili and G. Tutberidze,  Some problems of convergence of general Fourier series, \textit{J. Contemp. Math. Anal.}, 57 (2022), 6, {367}--{377}.
		
		\bibitem{tsagareishvili2022lipvar}
		V. Tsagareishvili,  The variation of general Fourier coefficients of Lipschitz class functions, \textit{Acta Math. Hungar.}, 166 (2022), 1, {142}--{161}.
		
		\bibitem{tsagareishvili2022slovaca}
		V. Tsagareishvili,  Lipschitz class functions and their general Fourier coefficients, \textit{Math. Slovaca}, 72 (2022), 3, {709}--{718}.
		
	\end{thebibliography}
\end{document}